\documentclass[12pt]{amsart}
\usepackage{amssymb}
\usepackage{color}   
 
\begin{document}

 \makeatletter
 	\newcommand{\sumprime}{\if@display\sideset{}{'}\sum%
 		\else\sum'\fi}
 	\makeatother
 	
 	\numberwithin{equation}{section}
 	
 	\newtheorem{theorem}{Theorem}[section]
 	\newtheorem{proposition}[theorem]{Proposition}
 	\newtheorem{corollary}[theorem]{Corollary}
 	\newtheorem{lemma}[theorem]{Lemma}

 	\theoremstyle{definition}
 	\newtheorem{definition}[theorem]{Definition}

 	\newtheorem{remark}[theorem]{Remark}
 	\newtheorem{question}{Question}
 	\def\thequestion{\unskip}
 	\newtheorem{example}{Example}
 	\def\theexample{\unskip}
 	\newtheorem{problem}{Problem}

\newcommand\bound{{\rm b}\,}
\newcommand\CC{{\mathbf C}}
\newcommand\Chi{\mathcal X}
\newcommand\dist{{\rm dist}\,}
\newcommand{\db}{\ensuremath{\overline\partial}} 
\newcommand\Dom {{\rm Dom\,}}
\newcommand\eps{\varepsilon}
\newcommand\Ker{{\rm Ker\,}}
\newcommand\length{{\ell}\,}
\newcommand\NN{{\mathbf N}}
\def\p{\partial}
\newcommand\Range{{\rm Range\,}}
\newcommand\RR{{\mathbf R}}
\newcommand\T{{\mathbf T}}
\newcommand\ZZ{{\mathbf Z}}

\title[Extendability and the $\db$ operator on the Hartogs triangle]
{Extendability and the $\db$ operator \\ on the Hartogs triangle} 
\author{Almut Burchard, Joshua Flynn, 
Guozhen Lu and Mei-Chi Shaw}
\address{Department of Mathematics, University of Toronto, Toronto, Ontario, Canada M5S 2E4}
 	\email{almut@math.toronto.edu}
	\address{Department of Mathematics, University of Connecticut, Storrs, Connecticut 06290}
 	\email{JOSHUA.FLYNN@UCONN.EDU}
	\address{Department of Mathematics, University of Connecticut, Storrs, Connecticut 06290}
 	\email{guozhen.lu@uconn.edu}
 	\address{Department of Mathematics, University of Notre Dame, Notre Dame, IN 46556}
 	\email{Mei-Chi.Shaw.1@nd.edu} 
 
\thanks
 	{A. Burchard  is partially supported by an NSERC discovery grant.  G. Lu and J. Flynn  are partially supported by the Simons foundation.
	M.-C. Shaw  is partially supported  by NSF grants.
 A. Burchard and M.-C. Shaw would like to  thank
the Banff International Research Station 
for  its kind  hospitality during a 2019 workshop  
which facilitated  this collaboration. We would also like to thank Christine Laurent-Thi\'ebaut for her helpful comments.
}

\begin{abstract} 
  In this paper it is shown that the Hartogs triangle $\T$ in $\CC^2$ is a uniform domain.
  This implies that the Hartogs triangle is a Sobolev extension domain.
  Furthermore, the weak and strong maximal extensions of the Cauchy-Riemann operator agree on the Hartogs triangle.
  These results have numerous  applications. Among other things, they  are used  to study the Dolbeault 
  cohomology groups with Sobolev coefficients 
  on the complement of $\T$.  
\end{abstract} 	


\newpage 
\maketitle

 \noindent{{\sc Mathematics Subject Classification} (2020): 32W05, 35N15.}
 
\section{\bf Introduction} 

  The Hartogs triangle $\T=\{(z,w)\in \CC^2 \mid |z|<|w|<1\}$  
is an important example in several complex variables. 
It is biholomorphic to the product of the 
unit disc with the punctured disc, hence a pseudoconvex 
domain and also a domain of holomorphy. 
However, it admits 
neither a Stein neighborhood basis nor a bounded plurisubharmonic 
exhaustion function.   The Hartogs triangle
plays an important role in our understanding of 
function theory for pseudoconvex domains (see the survey paper \cite{Shaw15}). 
There has been considerable interest in
the Bergman projection (see \cite{ChakrabartiShaw13, Shaw15})  
and the $\db$ problem on the Hartogs  triangle   
 (see e.g.,  \cite{ChakrabartiShaw13, ChakrabartiZeytuncu16, ChenMcNeal20, CC,  GGLV20,   
LaurentShaw13, LaurentShaw19, Sibony}), 
but  many fundamental  questions remain to be answered 
for this important model domain.

The Hartogs triangle is not a Lipschitz domain since
it is not the graph of any function near $(0,0)$. 
This presents  a substantial obstacle to the study of function 
theory on $\T$.  In this paper we show that $\T$ enjoys a number of 
properties generally associated with Lipschitz 
domains (see \cite{Stein70,Evans-Gariepy}).  Our first result
is that the Hartogs triangle is an  extension domain 
for  Sobolev spaces (Theorem \ref{th:extension}).
Consequently, smooth functions on $\CC^2$ are dense in 
the Sobolev spaces on $\T$, and the Sobolev embedding theorems hold.

Our main result concerns the Cauchy-Riemann operator
on the Hartogs triangle.  A fundamental tool for solving
the $\db$-problem for forms on any pseudoconvex domain
is H\"ormander's $L^2$ existence theory, where
the $\db$ operator is  defined in the weak maximal sense.
Specifically, for every weakly closed form $f$
on $\T$ with coefficients in $L^2$, the 
equation $\db u=f$ admits a weak solution in $L^2$.
The $L^2$-theory does not easily yield
information about the regularity of $u$,
even when $f$ is smooth up to the boundary of the domain.

There is another closed extension of the Cauchy-Riemann operator,
known as the strong maximal extension~$\db_s$ 
(see Definition \ref{de:dbar}) which is the 
closure of forms smooth up to the boundary 
in the $L^2$ graph norm.  The strong maximal extension
was used  by Kohn \cite{Kohn63, Kohn64} and 
Morrey \cite{Morrey66} in their approach to 
the $\db$-Neumann problem. If the domain is smooth and 
strongly pseudoconvex, and $f$ is $\db$-closed and  smooth up to the boundary, 
the Morrey-Kohn approach yields solutions that
are   smooth up to the boundary.  The strong extension $\db_s$ 
and its dual operator have many applications.
For  bounded domains with 
Lipschitz boundary, the equality of $\db$ and $\db_s$
was proved by H\"ormander 
from the Friedrichs lemma (see  \cite[Chapter 1]{Hormander65}). 

In this paper we prove that the weak and strong maximal
extensions agree on the Hartogs triangle. This is a first step
towards understanding the regularity of
solutions of the $\db$-problem on~$\T$. 
Since the Hartogs triangle is not a Lipschitz domain, 
the classical Friedrichs lemma does not apply, and
the relationship between the weak and strong extensions
of $\db$ is more subtle. 
Our results are based on the $L^2$ Serre duality that
relates the $\db$-Neumann problem to 
the $\db$-Cauchy problem.

\smallskip
 The plan of the paper is as follows: In Section 2, we 
study the Sobolev space $W^1(\T)$, consisting of $L^2$ 
functions with weak derivatives in $L^2$.
In the past few decades, there  has been tremendous 
progress in harmonic analysis on domains which are 
not Lipschitz, yet share some of their properties
(see   \cite{JK1982}  or the more recent paper \cite{ChordArc2017}).  
We show   that the Hartogs 
triangle is such a domain. Specifically we  prove 
(Theorem \ref{th:uniform}) that the Hartogs triangle 
is  a uniform domain in the sense 
of \cite{GO1979, MS79}. By a result of Jones~\cite{Jones81}, 
any uniform domain is a Sobolev extension domain. 
It follows that smooth functions are dense in 
the Sobolev spaces $W^k(\T)$ (Corollary~\ref{co:Sobolev}),
which in turn yields the Rellich compactness lemma and 
the Poincar\'e inequality on $W^1(\T)$.
Furthermore, we show that the Hartogs triangle has 
Ahlfors-David regular boundary
(Lemma~\ref{le:ADR}), and  the trace theorem  holds 
(Corollary \ref{co:trace}).  
These properties are used 
in our study of the $\db$ operator in later sections.

Section~3 is dedicated to the proof that
$\db=\db_s$ on $\T$ (Theorem~\ref{th:W=S}).
We take advantage of a number of recent
results on the properties of $\db_s$ on $\T$.
Following~\cite{LaurentShaw13}, we use $L^2$ Serre duality to
relate $\db$ and $\db_s$ to two
other closures of $\db$, namely the strong minimal 
closure $\db_c$,
and the weak minimal closure $\db_{\tilde c}$ 
(see Definition \ref{de:dbar c}). 
It was proved in~\cite{LaurentShaw13}  that 
for a rectifiable domain, the dual of $\db_s$ is  
$\db_{\tilde c}$.
Solving $\db_{\tilde c}$ amounts to solving 
$\db$ with prescribed support.
This is known as the $\db$-Cauchy problem and has 
numerous applications (see \cite{Shaw10, ChenShaw01}).  

The results from Section~2  are used 
to analyze the operators $\db_s$ and $\db_{\tilde c}$
on functions. It turns out that the kernel of $\db_s$ 
equals the kernel of $\db$, given by
the Bergman space of square integrable harmonic functions
on $\T$ (Proposition \ref{prop:Ker}).
In the proof, we explicitly estimate the terms
in the Laurent expansion on the Bergman space
from~\cite{ChakrabartiShaw13, Shaw15}.
Furthermore, $\db_c$ and $\db_{\tilde c}$ agree
on functions (Proposition~\ref{prop:W=S 2}). 

For $(0,1)$-forms, it is difficult to prove directly
that the kernel of $\db$ and kernel  $\db_s$  are the same 
since no simple Laurent expansion is available.
Instead, we use duality to show that
$\db_s$ has closed range (Theorem  \ref{th:closed strong}) and the relevant cohomology group is trivial.
Using duality once more, we conclude that 
$\db_s=\db$ (see Theorem~\ref{th:W=S}).

 In Section 4, we  study   solutions of $\db$ on  an  annular
domain between a pseudoconvex domain and the Hartogs triangle $\T$. 
Using the Sobolev extension theorem, we prove  that 
the $W^1$ Dolbeault cohomology on the annulus is isomorphic to the Bergman space on $\T$ (see Theorem \ref{th:W1 L2}).
This is in contrast with  the non-Hausdorff property   for 
the classical Dolbeault cohomology group for (0,1)-forms 
on the annulus  between a pseudoconvex  and the Hartogs
triangle   obtained earlier (see \cite{Trapani86} or \cite{LaurentShaw13}). 

There  remain many  open problems related to the Hartogs triangle 
which are  yet to be explored.  In Section 5 we 
present a number of problems.
It is still not known if the $L^2$ Dolbeault  cohomology 
on the annulus is  Hausdorff (Problem 1). Using the 
Sobolev extension theorem, this problem  is equivalent to asking 
if one can solve $\db$ in the Sobolev space $W^1(\T)$. One can 
ask more generally if one can solve 
$\db$ in any Sobolev space $W^s(\T)$ (Problem 2).  
The  Hodge theorem for the de Rham complex $d$ on $\T$ 
is also unknown for forms (Problem 3). Finally, one 
would also like to understand  the spectrum of the 
$\db$-Laplacian (Problem 4).  We also refer the 
reader to the many open problems for the Hartogs triangle
in $\mathbf{CP}^2$ in \cite{LaurentShaw19}. The
results in this paper  are only the beginning of 
understanding function theory 
on non-smooth domains.

\section{\bf Sobolev spaces on $\T$}

In this section, we establish some basic facts about the Sobolev
spaces $W^{1,p}$ on the Hartogs triangle
\begin{equation}
\label{eq:def-T}
\T:=\left\{(z,w)\in \CC^2\ \big\vert\ 
|z|<|w|<1\right\}\,.
\end{equation}
We show that these spaces have many useful properties, 
including bounded extension
and trace operators, smooth approximation,
Sobolev embeddings, and the Poincar\'e inequality.
The key to these properties are two geometric regularity results:
$\T$ is a uniform domain (Theorem~\ref{th:uniform}), 
whose boundary is Ahlfors-David regular
(Lemma~\ref{le:ADR}). 

\subsection{Uniform domain} 

\begin{definition}\label{de:uniform} 
Let $\Omega$ be a 
domain in $\RR^n$. The domain $\Omega$ 
is called an {\em $(\epsilon,\delta)$ domain} if 
for every $p_1,\ p_2\in \Omega$ and $|p_1-p_2|<\delta$, there 
exists a rectifiable curve $\gamma\in \Omega$ joining $x$ and $y$  such that 
$$\ell(\gamma)\le \frac 1\eps |p_1-p_2|$$
and 
$${\rm dist}(p, b\Omega) \ge\frac{ \eps  |p-p_1|  |p-p_2|} {|p_1-p_2|}\quad  
\text{for all } p\in \gamma.$$ 
where $\ell(\gamma)$ denotes the Euclidean  length of $\gamma$ 
and ${\dist}(p,b\Omega)$ denotes the distance from $p$ to $b\Omega$.

When $\delta=\infty$, $\Omega$  is called a {\em uniform domain}. 
 \end{definition}

Uniform domains were first introduced in \cite{MS79} 
and  \cite{GO1979}, while the notion of a $(\eps,\delta)$ 
domain  was introduced in \cite{Jones81}.  It 
turns out that for  bounded domains, they are 
equivalent~(see \cite{Vai1988}).


We will prove by direct computation that the Hartogs triangle
$\T$ is a uniform domain.  Following \cite{Jones81},  
it suffices  to show that 
there exists a constant $c>0$ such that
  every pair of points $p_1\ne\ p_2\in \T$ can be joined
  by a rectifiable curve $\gamma$ in $\T$ with
  \begin{equation}
    \label{eq:length}
    \length(\gamma)\le c|p_1-p_2|\,,
  \end{equation}
such that
\begin{equation} 
    \label{eq:dist}
    \min \{|p-p_1|, |p-p_2|\} \le c\, 
    {\rm dist}\, (p, b\T)
    \quad  \text{for all } p\in \gamma\,.
  \end{equation}
 
We begin with an elementary inequality.
\begin{lemma} [Distance in $\CC^2$]
\label{le:polar}
Let $p_1,p_2\in \CC^2$ be given by
$p_j=(r_{j}e^{i\alpha_j},s_{j}e^{i\beta_j})$,
where $r_{j},s_{j}\ge 0$ and $\alpha_j,\beta_j\in\RR$, j=1,2. 
If
$|\alpha_1\!-\!\alpha_2|\le \pi$ and $|\beta_1\!-\!\beta_2|\le \pi$, 
then
$$ 
|r_1\!-\!r_2|+|s_1\!-\!s_2|
+\min\{r_1,r_2\}|\alpha_1\!-\!\alpha_2| +\min\{s_1,s_2\}
|\beta_1\!-\!\beta_2|\ \le \ 3|p_1-p_2|\,.
$$ 
\end{lemma}

\begin{proof} 
We compare the right hand side of the inequality to
the squared distance
$$
|p_1-p_2|^2 
= |r_1-r_2|^2 + |s_1-s_2|^2 
+ r_1r_2|e^{i\alpha_1}-e^{i\alpha_2}|^2 
+ s_1s_2 |e^{i\beta_1}-e^{i\beta_2}|^2\,.
$$
Since $|\alpha_1-\alpha_2|\le \pi$,
their difference is comparable
to the distance of the corresponding unit vectors in $\CC$, 
$$
|\alpha_1-\alpha_2|\ \le \ 
\tfrac{\pi}{2} |e^{i\alpha_1}-e^{i\alpha_2}|\,,
$$
and correspondingly for $|\beta_1-\beta_2|$.
By Schwarz' inequality,
\begin{align*}
&
\Bigl(|r_1-r_2|+|s_1-s_2|
+\min\{r_1,r_2\} |\alpha_1-\alpha_2|
+ \min\{s_1,s_2\} |\beta_1-\beta_2|\Bigr)^2\\
&\qquad \le 
\Bigl(|r_1-r_2|+|s_1-s_2|+ \tfrac{\pi}{2}
\sqrt{r_1r_2}|e^{i\alpha_1}-e^{i\alpha_2}|
+\tfrac{\pi}{2} \sqrt{s_1s_2}|e^{i\beta_1}-e^{-\beta_2}|\Bigr)^2\\
&\qquad \le \left(2+\tfrac{\pi^2}{2}\right)
|p_1-p_2|^2\,.
\end{align*}
Since $\sqrt{2+\frac{\pi^2}{2}}<3$, this proves the claim.
\end{proof}

In order to understand the role of the singularity of $\T$
the origin, we consider the infinite Hartogs triangle
\begin{equation}
\label{eq:def-T infty}
\T_\infty:= \left\{(z,w)\in\CC^2\ \big\vert\ |z|<|w|\right\}\,.
\end{equation}

\begin{lemma} 
  \label{le:T-infty-uniform} $\T_\infty$ is a uniform domain.
\end{lemma}

\begin{proof} We will join any given pair
of points $p_1\ne p_2\in \T_\infty$
by a curve $\gamma$ in $\T_\infty$ that satisfies
 Eqs.~\eqref{eq:length} and~\eqref{eq:dist} with 
$c=5+2\pi<12$.
The curve consists of an arc $\gamma_0$ that maintains
a constant distance from the boundary, and
a pair of line segments $\gamma_1$ and $\gamma_2$
attached at the ends.

For $j=1,2$, write the points in polar coordinates
as $p_j=\bigl(r_j e^{i\alpha_j}, s_j e^{i\beta_j}\bigr)$
with $0\le r_j<s_j$, $|\alpha_1\!-\!\alpha_2|\le \pi$,
  and $|\beta_1\!-\!\beta_2|\le \pi$.
Choose $\gamma_0$ as the arc parametrized by 
  $ p=\bigl(r_* e^{i\alpha}, s^* e^{i\beta}\bigr)$, where
\begin{equation}
\label{eq:rs}
  r_*:= \min\{r_1,r_2\}\,,\qquad
  s^*:=\max\{s_1,s_2\}+ |p_1-p_2|\,,
\end{equation}
and the angles vary linearly from $\alpha_1,\beta_1$ to 
$\alpha_2,\beta_2$. Its endpoints 
\begin{equation}
\label{eq:qj}
q_j:=\bigl(r_* e^{i\alpha_j}, s^* e^{i\beta_j}\bigr)\,, \quad
j=1,2
\end{equation}
are joined to the corresponding 
points $p_j$ by line segments $\gamma_j$.

  \smallskip\noindent {\em Length of the curve.}\ 
  Since $s^*\le \min\{s_1,s_2\}+ 2|p_1-p_2|$,
  we obtain for the arc
  $$
  \length(\gamma_0)
  \ \le \ \min\{r_1,r_2\}|\alpha_1-\alpha_2|
  + \min\{s_1,s_2\}|\beta_1-\beta_2|
  +2\pi |p_1-p_2| \,.
  $$
  The initial and final segments satisfy
  \begin{align*}
    \length(\gamma_1) 
    & \le \ (r_1-r_2)_+ + (s_1-s_2)_- +|p_1-p_2|\,,\\
    \length(\gamma_2) 
    & \le \ (r_1-r_2)_- + (s_1-s_2)_+ +|p_1-p_2|\,,
  \end{align*}
  where $x_{+}$ and $x_{-}$ denote, respectively, the positive and negative parts of a number $x \in \RR$.
Adding the three inequalities yields by
Lemma~\ref{le:polar} that 
$\length(\gamma) \le \bigl(5+2\pi\bigr)|p_1-p_2|$.

  \small\noindent{\em Distance from the boundary.}\ 
Let $p=(re^{i\alpha}, se^{i\beta})\in \T_\infty$.
Since the ball of radius $\frac{1}{\sqrt2}|r-s|$ about $p$
is contained in $\T_\infty$, and its boundary
meets $b\T_\infty$ in the single point 
$\tfrac{r+s}{2}(e^{i\alpha},e^{i\beta})$, 
  $$
  {\rm dist}\, (p,b\T_\infty)= \left|p-\tfrac{r+s}{2}(e^{i\alpha},e^{i\beta})
  \right|=\tfrac1{\sqrt2}|s-r|\,.
$$
On the central arc, the distance from the boundary
is constant,
$$
{\rm dist}\,(p,b\T_\infty) \ =\ \tfrac1{\sqrt2}(s^*-r_*)
  \ \ge \ \tfrac1{\sqrt2} |p_1-p_2|\,,\quad p\in \gamma_0\,.
$$
Since $\min\{|p-p_1|,|p-p_2|\}\le \frac12 \ell(\gamma)$,
it follows that
$$
\frac{\min\{|p-p_1|,|p-p_2|\}} {{\rm dist}\,(p,b\T_\infty)}
\ \le \ \frac{5+2\pi}{\sqrt2}\,,
\quad p\in \gamma_0\,.
$$
This yields Eq.~\eqref{eq:dist} on $\gamma_0$.
On the segments $\gamma_j$, the distance from the boundary increases
linearly (in the arclength parametrization) 
from $p_j$ to $q_j$, and therefore 
$$
\frac{|p-p_j|}{
{\rm dist}\, (p,b\T_\infty)}
\ \le \ \frac{|q_j-p_j|}{{\rm dist}\, (q_j,b\T_\infty)}\,,
\qquad p\in\gamma_j,\ j=1,2\,.
$$
Since $q_1,q_2$ lie on the arc $\gamma_0$, this completes the proof
of Eq.~\eqref{eq:dist}.
\end{proof}

  \begin{theorem}\label{th:uniform} The Hartogs triangle ${\T}$  
is a uniform domain.
 \end{theorem} 

\begin{proof} 
  Given two points $p_1,p_2\in \T$, define
the radii $r_*$, $s^*$ by Eq.~\eqref{eq:rs},
and let $\gamma'=\gamma'_0\cup\gamma'_1\cup\gamma'_2$ 
 be the curve constructed in the proof of 
Lemma~\ref{le:T-infty-uniform}.
We rescale the arc to
$$
\gamma_0:=\frac{1}{1\!+\!2|p_1-p_2|}\gamma'_0\,,
$$ 
and then join its endpoints $q_j$  
to $p_j$ by line segments $\gamma_j$, for $j=1,2$.
  Since $r_*<s^*<1\!+\!|p_1-p_2|$,
  the curve $\gamma:= \gamma_0\cup \gamma_1\cup \gamma_2$
  lies in $\T$. We will show that $\gamma$ satisfies
  Eqs.~\eqref{eq:length} and~\eqref{eq:dist}
  with $c=(1+4\sqrt{2})(5+2\pi +4\sqrt{2})/\sqrt{2}<80$.

\smallskip\noindent{\em Length of the curve.}\ 
By construction, $\length(\gamma_0) \le \length(\gamma'_0)$.
For $j=1,2$, we write the
endpoints of $\gamma_0$
as $q_j= (1+2|p_1-p_2|)^{-1} q_j'$, where
$q_j'$ is given by Eq.~\eqref{eq:qj}.  
Since 
$$|q_j'-q_j|\le 2|q_j|\,|p_1-p_2|\quad\text{and}\quad |q_j|\le \sqrt{2}\,,
$$
by the triangle inequality the lengths of the line segments satisfy
$$
\ell(\gamma_j) \le|p_j-q_j'|+ |q_j'-q_j| 
\le \length(\gamma')+ 2\sqrt{2}|p_1-p_2|\,,  \quad j=1,2\,.
$$
Adding these inequalities yields, by
Lemma~\ref{le:T-infty-uniform},
$$
\length(\gamma) \le \ell(\gamma')+ 4\sqrt{2}|p_1-p_2|
\le (5+ 2\pi + 4\sqrt{2})\, |p_1-p_2|\,.
$$

  \smallskip\noindent{\em Distance from the boundary.}\ 
For $p=(re^{i\alpha}, se^{i\beta}) \in \T$, we have
  $$
  {\rm dist}\, (p,b\T)= \min\left\{
  \tfrac1{\sqrt2} (s-r), 1\!-\!s\right\}\,,\quad p\in\T\,.
  $$
On the central arc $\gamma_0$, this distance constant,
\begin{align*}
  {\rm dist}\, (p,b\T)
  &=\min \left\{ \frac1{\sqrt2} \left(\frac {s^*-r_*} 
  {1\!+\!2|p_1-p_2|}\right), 
  1- \frac{s^*}{1\!+\!2|p_1-p_2|}\right\}\\
& \ge\frac{|p_1-p_2|}{\sqrt{2}(1+4\sqrt{2})}\,,\qquad p\in\gamma_0\,.
\end{align*}
In the second step we have used that
$|p_1-p_2| \le {\rm diam}\, (\T)=2\sqrt{2}$ 
to bound the denominators from above,
and that $s^*<1+|p_1-p_2|$ to bound the second
fraction from below.

Since $\min\{|p-p_1|,|p-p_2|\}\le\frac12\ell(\gamma)$ 
for every $p\in\gamma$, it follows that
$$
\frac{\min\{|p-p_1|,|p-p_2|\}}
{{\rm dist}\, (p,b\T)}
\ \le \ \frac{(1+4\sqrt{2})(5+2\pi + 4\sqrt{2})}{\sqrt{2}}
  \,,\quad p\in\gamma_0\,.
$$
This yields Eq.~\eqref{eq:dist} on $\gamma_0$.
  On $\gamma_j$, we argue as in the proof of 
Lemma~\ref{le:T-infty-uniform}
that the distances from the parts of the boundary at $\{z=w\}$ and 
$|w|=1$ change linearly along the segments to obtain
  $$
\frac{|p-p_j|}{{\rm dist}\, (p,b\T)} 
\ \le \ \frac{|q_j-p_j|}{{\rm dist}\, (q_j,b\T)} \,,\qquad
p\in\gamma_j, \ j=1,2\,.
$$
Since $q_1,q_2\in\gamma_0$, this completes the proof.
\end{proof}

For later use, we briefly discuss domains in the
complement of $\T$.

\begin{lemma}
\label{le:complement} 
Let $\Omega\subset \CC^2$ be a bounded Lipschitz domain
with $\Omega\supset \overline{\T}$. 
Then $\Omega\setminus\overline\T$ is a uniform domain.
\end{lemma} 

\begin{proof} 
It suffices to prove that there exist constants 
$\delta>0$ and $c>0$ such that
any pair of points $p_1,p_2\in \Omega\setminus \T$ 
with $|p_1-p_2|<\delta$ can be joined by a curve 
$\gamma$ in $\Omega\setminus \T$ that satisfies
Eqs.~\eqref{eq:length} and~\eqref{eq:dist}.

For $\rho>0$, denote by 
$D^2_{\rho}:= \{(z,w)\in\CC^2: |z|<\rho, |w|<\rho\}$
the bidisk of radius~$\rho$.
Using that $\Omega$ contains the origin, 
we choose $\rho\in (0,1)$ so small that
$D^2_{\rho}\subset\Omega$, and write
$\Omega= O_1\cup O_2$,
where
$$
O_1:= D^2_{\rho}\,,
\quad O_2:= \Omega\setminus \overline{D^2_{\rho/2}}\,.
$$
Then $O_1\setminus\overline{\T}= 
\{(z,w)\in\CC^2\ \big\vert\ |w|<|z|<\rho\}$ 
is a scaled copy of $\T$, and
$O_2\setminus\overline\T$ has Lipschitz boundary 
since $b D^2_{\rho/2}$ intersects
$b\T$ transversally. Thus both are uniform
domains.  For $\delta>0$ sufficiently small,
any pair of points in $\Omega\setminus\overline{T}$
with $|p_1-p_2|<\delta$ is contained
in one of the sets $\Omega_j$, in such a
way that the distance of both  points from
$b\T$ is comparable to the distance from 
$\Omega\setminus \overline\T$.  Therefore, a 
connecting curve satisfying Eq.~\eqref{eq:dist} can be constructed
in either $O_1\setminus\overline\T$
or $O_2\setminus\overline\T$.
\end{proof}

\subsection{Boundary regularity}

\begin{definition}\label{de:ADR} Let $\sigma$ denote
the $(d-1)$-dimensional Hausdorff measure on $\RR^d$.
A~closed subset
$S\subset\RR^d$ is {\em Ahlfors-David regular}, if
there exists a constant $c>0$ such that
\begin{equation}
\label{eq:ADR}
c^{-1}\rho^{d-1}\ \le\  \sigma \bigl(B_\rho(p)\cap S)\bigr) 
\ \le \ c \rho^{d-1}
\end{equation}
for all $p\in S$ and all $\rho<{\rm diam}\, S$. Here, $B_\rho(p)$ denotes
the ball of radius $\rho$ about $p$.
\end{definition}

We first consider the unbounded Hartogs triangle from
Eq.~\eqref{eq:def-T infty}. Clearly, $b\T_\infty$ is rectifiable
with respect to the three-dimensional Hausdorff
measure, $\sigma$. In particular, the restriction of $\sigma$
to $b\T_\infty$ agrees with the natural surface measure.

\begin{lemma} 
\label{le:sigma-continuous}
$\sigma \bigl(B_\rho(p)\cap b\T_\infty)\bigr)$ 
depends continuously on $\rho\ge 0$ and $p\in b\T_\infty$.
\end{lemma}
\begin{proof} Consider first the dependence on the radius.
For every fixed $p\in b\T_\infty$, the
function $g(\rho):= \sigma \bigl(B_\rho(p)\cap b\T_\infty)\bigr) =0$
is non-decreasing in $\rho$. Hence $g$ has at most countably 
many discontinuities, given by jumps. Since $g$ is left continuous,
the size of the jump at $\rho$ equals
$$
g(\rho_+)-g(\rho)=
\sigma \bigl(bB_\rho(p)\cap b\T_\infty\bigr) \,. 
$$ 

For $p\in b\T_\infty$ and $\rho\ne |p|$,
the sphere $bB_\rho(p)$ intersects
$b\T_\infty$ transversally.
Indeed, if a sphere were to touch $b\T_\infty$ at a point $q=(z,w)\ne 0$, 
then its center would lie on the line through $q$ 
normal to the surface. This normal line has the form
$\{(z,w)+s(z,-w): s\in\RR\}$, which does
not intersect $b\T_\infty$ again. 
By the Implicit Function Theorem,
the transversal intersection 
$bB_\rho(p)\cap b\T_\infty$ is a submanifold of dimension 2.
On the other hand, for $\rho=|p|>0$ the intersection
$bB_\rho(p)\cap b\T_\infty$ contains the singular point
at the origin; away from the origin the intersection 
is again transversal.
In any case, the three-dimensional Hausdorff measure of 
the intersection vanishes, establishing continuity in $\rho$.

We turn to the dependence on $p$. For $q\in b\T_\infty$
with $|p-q|<\delta$, we have that
$$
B_{\rho-\delta}(p)\ \subset \ B_\rho(q)\ \subset \ B_{\rho+\delta}(p)\,,
$$
and hence, by the positivity of $\sigma$,
$$
\sigma(B_{\rho-\delta}(p)\cap b\T_\infty) 
\ \le\ \sigma(B_\rho(q)\cap b\T_\infty)
\ \le \ \sigma(B_{\rho+\delta}(p)\cap b\T_\infty)\,.
$$
Continuity in $p$ thus follows from continuity in $\rho$.
\end{proof}

\begin{lemma} 
  \label{le:T-infty-ADR}
$\T_\infty$ has Ahlfors-David regular boundary.
\end{lemma}

\begin{proof} 
We parametrize the boundary $b\T_\infty$ by
$$
p(r,\alpha,\beta):= \tfrac{1}{\sqrt{2}}(re^{i\alpha},re^{i\beta})\,,
\qquad r\ge 0\,, \alpha,\beta\in (-\pi,\pi]\,.
$$
The Jacobian of the parametrization
equals $\frac12r^2$.

For $t\ge 0$, let
$$
f(t) := \ \sigma \bigl(B_1(p(t,0,0))\cap b\T_\infty\bigr) 
$$
be the surface measure of the
intersection of the boundary with a ball of 
unit radius centered at $\frac{1}{\sqrt{2}}(t,t)\in b\T_\infty$.
By the rotation and dilation invariance of $\T_\infty$,
$$
\sigma \bigl(B_\rho(p)\cap b\T_\infty)\bigr) = \rho^3 
    f\left( \rho^{-1}|p|\right)\,,
\qquad \rho>0\, , p\in b\T_\infty\,.
$$
In the chosen parametrization,
$$
f(t)\ = \ \frac12 \int_{-\pi}^\pi \int_{-\pi}^\pi\int_0^\infty
\mathbf{1}_{B_1(p(t,0,0))}\bigl(p(r,\alpha,\beta)\bigr)\,
r^2 \, dr d\alpha d\beta\,.
$$
The  intersection is described by the inequality
$$ 
p(r,\alpha,\beta)\in B_1(p(t,0,0)) \ 
\Leftrightarrow\ (r-t)^2+2rt(\sin^2\tfrac{\alpha}{2}
+\sin^2 \tfrac{\beta}{2})< 1\,.
$$ 
For $t=0$, the condition simplifies to $r^2<1$, and
we find that
$$
f(0) = \frac12 \int_{-\pi}^\pi \int_{-\pi}^\pi\int_0^1
r^2 \, dr d\alpha d\beta = \frac{2\pi^2}{3}\,.
$$
For larger $t$, the condition is  
$\sin^2\tfrac{\alpha}{2} +\sin^2\tfrac{\beta}{2}
< \frac{1-(r-t)^2}{2tr}$,
and hence $\alpha, \beta = O(t^{-1})$ as $t\to\infty$.
Taylor expansion of the sine yields
$$
\alpha^2+\beta^2\le \frac{2(1-(r-t)^2)}{tr} + O(t^{-4})\,.
$$
For any $r>0$ with $|r-t|<1$, this inequality
defines an approximate disk in the $\alpha$-$\beta$ variables.
After integrating out these variables,
we are left with
$$
f(t)  \ = \ 
\pi \int_{t-1}^{t+1}\frac{1-(r-t)^2}{tr} r^2\, dr + O(t^{-2})
\quad \to \quad \frac{4\pi}{3} \qquad (t\to\infty)\,.
$$ 
In the last step, we have evaluated the integral explicitly.
Since $f$ is continuous by Lemma~\ref{le:sigma-continuous},
as well as strictly positive,
it follows that
$$
\inf_{t\ge 0} f(t)>0\,,\qquad \sup_{t\ge 0} f(t)<\infty\,,
$$
proving the claim.
\end{proof}

\begin{lemma} 
  \label{le:ADR}
The Hartogs triangle $\T$ has Ahlfors-David regular boundary.
\end{lemma}

\begin{proof}
We need to  verify Eq.~\eqref{eq:ADR}  
for all $p\in b\T$ and all $\rho\in (0,2]$.
By rotational symmetry, we may assume that $p=(r,s)$
with $0\le r\le s\le 1$. 

We partition the boundary as 
$$ b\T
\ =\ \bigl(b\T_\infty \cap\{|w|\le 1\}\bigr) \ \cup \ 
\bigl(\T_\infty \cap\{|w|=1\}\bigr)\,.
$$
For the upper bound on the first term, if
$B_\rho(p)$ meets $b\T_\infty$ in a point 
$q\in b\T_\infty$ with $|p-q|<\rho$,
then $B_{\rho}(p)\cap b\T_\infty$ is contained in 
$B_{2\rho}(q)\cap b\T_\infty$. The measure of this ball
satisfies the desired bound by 
Lemma~\ref{le:T-infty-ADR}. Similarly, if $q\in B_\rho(p)\cap \{|w|=1\}$,
then we use that
$B_\rho(p)\cap \{|w|=1\} \subset B_{2\rho}(q)\cap\{|w|=1\}$,
which is contained in the product of a disk of radius $2\rho$
with a spherical arc of diameter at most $2\rho$. 
The surface measure of this intersection is at most 
comparable to $\rho^3$.

For the lower bound, we distinguish two cases.
If $p=(r,r)$ with $r\le 1$, then
$$
B_\rho(p)\cap b\T\  \supset \ 
\left\{(z,w)\in b\T_\infty\ \Big\vert\ 
r-\tfrac{\rho}{\sqrt{2}}< |z|=|w|<r  \right\} \,.
$$
The right hand side is homogeneous of
order three, as well as continuous and strictly positive.
As in the proof of Lemma~\ref{le:T-infty-ADR},
this implies  a lower bound of order $\rho^{3}$.
If, on the other hand, $p=(r,1)$ with $r<1$ then
\begin{align*}
B_\rho(p)\cap b\T
&\ \supset \ \left\{(z,w)\in\T_\infty \ \Big\vert\ 
|w|=1, |z-r|^2+|w-1|^2<\rho^2\right\} \\
& \ \supset \ 
\Bigl\{ |z|<1, |z-r|<\tfrac{\rho}{\sqrt{2}}
\Bigr\}
\ \times \ 
\Bigl\{|w|=1, |w-1|<\tfrac{\rho}{\sqrt{2}}\Bigr\}\,.
\end{align*}
Its surface measure is bounded from below
by a constant multiple of $\rho^3$.
\end{proof}

Our results can be summarized as follows.

\begin{corollary}\label{co:chord-arc} The Hartogs triangle $\T$
is a chord-arc domain.
\end{corollary}

\begin{proof} Both $\T$ and $\CC^2\setminus\overline\T$
are uniform domains, and $b\T$ is Ahlfors-David regular.
\end{proof}

\subsection{Sobolev theorems} 

Let $\Omega$ be a domain in $\RR^d$.  
For $k\in \NN$ and $1\le p\le \infty$,  denote
by $W^{k,p}(\Omega)$ the Sobolev  space of functions whose weak derivatives 
of order up to $k$ lie in $L^p$. 
When $p=2$, we also use $W^k(\Omega)$ to denote $W^{k,2}(\Omega)$.

 \begin{definition}\label{de:extension} A domain
$\Omega\subset \RR^d$ is called a {\em (Sobolev) extension domain},
if for each $k\in \NN$ 
and $1\le p\le \infty$, there exists a bounded linear operator 
$$\eta_k : W^{k,p} (\Omega)\to W^{k,p}(\RR^d)$$
such that $\eta_kf\big\vert_{\Omega} = f$ for all $f\in W^{k,p}(\Omega)$. 
 \end{definition}
 
It is well known that every bounded Lipschitz domain 
in $\RR^d$ is an extension domain. 
 Our main result in this section is the following:
 \begin{theorem}\label{th:extension} 
The Hartogs triangle ${\T}$ is a Sobolev extension domain.
\end{theorem}
 
\begin{proof} By Theorem \ref{th:uniform}, $\T$ is a 
uniform domain. Hence $\T$ is an $(\eps,\delta)$ domain with $\delta=\infty$. 
Using \cite[Theorem 1]{Jones81}, every $(\eps, \delta)$ 
domain is an extension domain.
\end{proof}

For later use, we state the corresponding result for
the complement of $\T$.

\begin{lemma}
[Extension from the complement of $\T$]
\label{le:extension annuli} Let 
$\Omega\subset \CC^2$ be a domain
with $\Omega\supset \overline{\T}$.
There exists a bounded linear operator 
$\eta  : W^1(\Omega\setminus\overline{\T})\to W^1(\Omega)$
such that $\eta f\big\vert_{\Omega\setminus\overline\T}=f$.
If, moreover, $\Omega$ is a Lipschitz domain, 
then $\Omega\setminus\overline{\T}$ is an extension domain. In particular, 
$C^\infty(\CC^2)$ is dense in $W^1(\Omega\setminus\overline\T)$.
\end{lemma} 

\begin{proof}  If $\Omega$ is a bounded Lipschitz domain,
then $\Omega\setminus\overline\T$ is
a uniform domain by Lemma~\ref{le:complement},
and hence a Sobolev extension domain.

For an  arbitrary domain $\Omega\supset\overline\T$, we choose 
a Lipschitz domain $\Omega'$ with
$\Omega\supset \overline{\Omega'}$ such that
$\Omega'\supset \overline\T$. By the first part of the proof,
there exists a bounded linear extension operator
$\eta':W^1(\Omega'\setminus\overline\T)\to W^1(\Omega')$.
For $f\in W^1(\Omega\setminus\overline\T)$, we define
$\eta f$ on $\overline{\T}$ by first restricting
$f$ to $\Omega'\setminus\T$ and then applying $\eta'$.
\end{proof}

From Theorem \ref{th:extension}, we  have the following 
results easily. 

\begin{corollary}\label{co:Sobolev} Let  $W^1(\T)$ denote the Sobolev 
space of $L^2$-functions on $\T$
with weak first-order derivatives 
in $L^2$.  Then the following statements hold:
\begin{enumerate}
\item {\em (Smooth approximation).}\ 
$C^\infty(\overline{\T})$ is dense in $W^1(\T)$.

\item {\em (Sobolev embedding).} \ $W^1(\T)\subset L^4(\T)$,
and the inclusion map is bounded.

\item {\em (Rellich lemma).}\ 
The inclusion $W^1(\T)\subset L^2(\T)$  is compact. 
\end{enumerate}
\end{corollary}

\begin{proof} 
Let $\eta:W^1(\T)\to W^1(\RR^4)$ be the bounded linear
extension operator provided by Theorem~\ref{th:extension}.
Given $f\in W^1(\T)$, set $f_0:=\eta f\in W^1(\CC^2)$.
We regularize $f$ by  convolution   $f_\eps=f_0\ast \phi_\eps$,
where $\{\phi_\eps\}$ is an approximate identity 
such that each $\phi_\eps$ is a smooth function of
compact support with $\int\phi_\eps=1$. 
The restrictions of the smooth 
functions $f_\eps$ to $\T$ converge to $f$ in $W^1(\T)$, 
proving the first claim. 
By the Sobolev inequality on $\CC^2$,
$\|\eta f\|_{L^4(\CC^2)}\le C \|\eta f\|_{W^1(\CC^2)}$,
where $C$ is the Sobolev constant.
Since $\eta f $ agrees with $f$ on $\T$, and
$\eta:W^1(\T)\to W^1(\CC^2)$ is bounded,
this implies the second claim.  Similarly, the Rellich lemma 
holds on $\T$ because it holds on a ball 
$\Omega_0\supset \T$.
\end{proof}

\begin{corollary}[\textbf{Trace}]\label{co:trace}
There  exists a bounded linear operator
$\tau: W^1(\T)\to L^2(b\T)$ with the property that
$\tau f =f\big\vert_{b\T}$ for
every continuous function $f$ on $\overline{\T}$.
\end{corollary}

\begin{proof}  
Since
$\T$ is a uniform domain with Ahlfors-David regular boundary,
the existence of the trace operator 
follows from~\cite[Theorem 3]{Jonsson79}.
\end{proof}

\begin{corollary}[\textbf{Poincar\'e inequality}]
\label{co:Poincare}
There exists a constant $C>0$ such that
$$ 
\|f\|^2 \le C \|d f \|^2 
$$ 
for all $f\in W^1(\T)$ with $(f,1)=0$, where $\|\ \|$ denotes the $L^2$-norm on $T$. 
\end{corollary}

\begin{proof} 
The proof is the same as for smooth  
domains.  Let
$$
\lambda:=\inf\left \{\|d u\|^2 \ \big\vert\ 
u\in W^{1}(\T), \|u\|=1, (u,1)=0\right\}\,
$$
and consider a minimizing sequence $\{u_\nu\}$.
By the Rellich lemma, we may assume
(after passing to a subsequence) that 
$u_\nu$ converges strongly in $L^2(\T)$, as well as weakly 
in $W^1(\T)$, to some limit $v\in W^1(\T)$.
The function $v$ is non-constant, because
$$
\|v\|=\lim \|u_\nu\|=1\,,\qquad (v,1) =\lim (u_\nu, 1)=0\,.
$$ 
Since $\|d v\|^2\le \lim \|d u_\nu\|^2=\lambda$
by the weak lower semicontinuity of the Dirichlet integral,
$v$ is a minimizer. 
Therefore $\lambda=\|d v\|^2>0$,
and the Poincar\'e inequality holds with $C=\lambda^{-1}$.
\end{proof}

\section{\bf Identity of weak and strong extensions of $\db$}

We collect some basic facts on 
unbounded operators in Hilbert spaces which will be used later. 
	
\subsection{Basic facts from functional analysis}	
We recall the definition of an unbounded linear  operator from a Hilbert space to another.  
By an  {\em operator} $A$ from a Hilbert space $ {H_1}$ to another Hilbert space
$\ {H_2}$ we mean a $\CC$-linear map from a linear subspace $\Dom(A)$ 
of $ {H_1}$ into ${H_2}$. We use the notation 
$$A: {H_1}\dashrightarrow {H_2},$$
to denote the fact that $A$ is defined on a subspace of $ {H_1}$.  Recall that  an
operator is said to be {\em closed} if its graph is 
closed as a subspace of the product
Hilbert space $ {H_1}\times  {H_2}$.
Suppose that $A$ is defined    on all of $ {H_1}$,
then we write $A: {H_1}\rightarrow  {H_2}$. Notice that  if $A$ is defined on the  whole Hilbert space $H_1$, then $A$ has to be a bounded  operator from the closed  graph theorem.

    Let $A$ be a closed densely 
defined operator from $H_1$ to $H_2$. 
Let   $A^*: H_2\dashrightarrow H_1$   be the (Hilbert space) 
adjoint of $A$, defined as follows: 
An element $g\in \Dom(A^*)$ if and only if there 
exists an element $g^*\in H_1$ such that
$$(Af, g)=(f, g^*)\qquad \text{for all }f\in \Dom(A)\,.$$
In this case, $A^*g:=g^*$. Then
$A^*$ is also a densely defined closed operator, and 
$A^{**}=A$.

  We refer the reader to the book of  Riesz-Nagy for 
Hilbert space adjoints (see page 305 in \cite{RieszNagy}).  
We will also  need  the following  lemma
    (see \cite[Theorem 1.1.1]{Hormander65}). 
    
    \begin{lemma}
\label{le:closed equivalent} 
Let $A$ be a closed densely defined operator from one Hilbert space $H_1$ to another $H_2$. Then the following 
    conditions are equivalent:
     \begin{enumerate}
    \item The range of $A$ is closed.
  \item The range of $A^*$ is closed.
  \item $H_1= \Range(A^*)\oplus \Ker(A).$
  \item $H_2= \Range(A)\oplus \Ker(A^*).$
  \end{enumerate}
  \end{lemma}
\begin{proof}
  For a proof that  (1) and (2) are equivalent, see \cite{Hormander65}.
    By definition of the adjoint operator, the range of $A$ is the closure 
of the orthogonal complement of 
    the kernel of $A^*$. Thus  if the range of $A$ is closed, then
it is the orthogonal complement of the kernel of $A^*$.  
Thus  (1) and (4) are equivalent. 
Similarly,  (2) and (3) are equivalent.
\end{proof}

\subsection{Maximal extensions and minimal closures of $\db$}

Let $\Omega$ be a bounded  domain in $\CC^n$.
Let 
\begin{equation}\label{eq:CR smooth} 
\db :C^\infty_{p,q-1}
(\overline\Omega) \to C^\infty_{p,q}(\overline \Omega) \,,
\qquad 0\le p\le n, 1\le q\le n
\end{equation}
be the classical Cauchy-Riemann operator on smooth forms. 
We will use the same symbol, $\db$, to denote
the weak Cauchy-Riemann operator acting on currents.
Since the index $p$ plays no role on $\CC^n$, we will 
make convenient choices for the value of $p$ in our arguments.

Let $L^2_{p,q}(\Omega)$   be the space of square-integrable 
$(p,q)$-forms on $\Omega$. The classical
$\db$ operator on smooth forms can be 
extended to a closed densely defined unbounded 
operator from $L^2_{p,q-1}({\Omega})$ to $L^2_{p,q}({\Omega})$
in several different ways.

\begin{definition} \label{de:dbar}
Let $\Omega$ be a bounded  domain in  $\CC^n$.  
\begin{enumerate}
\item The   {\it weak maximal} extension of $\db$, denoted by
$$
\db: L^2_{p,q-1}({\Omega}) \dashrightarrow L^2_{p,q}({\Omega})\,,
$$
is defined by $f\in \Dom(\db)\cap L^2_{p,q-1}(\Omega)$  
if and only if $\db f\in L^2_{p,q}(\Omega)$ in the distribution sense.  
\item 
 The   {\it strong maximal extension} of $\db$, denoted by
$$\db_s:L^2_{p,q-1}({\Omega})\dashrightarrow L^2_{p,q}({\Omega})\,,$$ 
is the closure in the graph norm of the restriction 
of $\db$ to the smooth forms in $C^\infty(\overline\Omega)$. 
In other words, $f\in \Dom(\db_s)$ if and only if there 
exists a sequence of smooth forms  $f_\nu$ 
in $C^\infty_{p,q-1}(\overline\Omega)$  
such that  $f_\nu\to f$ and $\db f_\nu\to \db f$ in $L^2$.  
\end{enumerate}
\end{definition}

It is easy to check (by smooth approximation) 
that if $f\in \Dom(\db_s)$, then 
$f\in \Dom(\db)$ and $\db  f=\db_s f$.
Hence $\db$ is a closed extension of $\db_s$.
On any bounded Lipschitz domain $\Omega$,
the Friedrichs lemma implies that $\db=\db_s$,
see H\"ormander  \cite{Hormander61,Hormander65} 
(or Lemma 4.3.2 in the book by Chen-Shaw \cite{ChenShaw01}).  

In addition to the maximal extensions, we will 
consider the following minimal closures of the 
Cauchy-Riemann operator.

\begin{definition}
\label{de:dbar c} 
Let $\Omega$ be a bounded  domain in  $\CC^n$.  
 \begin{enumerate}
\setcounter{enumi}{2}
\item The   {\it strong  minimal} closure of $\db$, denoted by
$$\db_c:L^2_{p,q-1}({\Omega})\dashrightarrow L^2_{p,q}({\Omega})$$ 
is the closure in the graph norm of the restriction of $\db$ 
to the smooth compactly supported forms. In other words,
  $f\in \Dom(\db_c)$ if and only if there is a 
sequence of smooth forms  $f_\nu$ in $C^\infty_{p,q-1}(\Omega)$ 
compactly supported in $\Omega$, 
such that  $f_\nu\to f$ and $\db f_\nu\to \db f$ in $L^2$.
 \item  The {\it weak minimal} closure, denoted by
$$\db_{\tilde c}:L^2_{p,q-1}({\Omega})\dashrightarrow L^2_{p,q}({\Omega}) $$
is defined by $f\in \Dom(\db_{\tilde c})\cap L^2_{p,q-1}(\Omega)$  
if and only if $\db f^0$ is in $L^2_{p,q-1}(\CC^2)$, 
where $f^0$ is the extension of $f$ to zero outside $\Omega$ and $\db f^0$ 
is defined in sense of distribution on $\CC^2$.   
\end{enumerate}   
   \end{definition}

As above, it is easy to check 
that $\db_{\tilde c}$ is a closed extension of $\db_c$.
In fact, 
$$ 
\Dom(\db)\ \supset\ \Dom(\db_s)\ \supsetneq\ \Dom(\db_{\tilde c})
\ \supset\ \Dom(\db_c)\,,
$$ 
and the corresponding inclusions hold for the ranges and the
kernels.
The middle inclusion is proper, because
non-zero constant functions lie in the domain
of $\db_s$, but not in the domain of $\db_{\tilde c}$.

\subsection{The operators $\db_s$ and $\db_{\tilde c}$ on functions}

We start with some simple observations about the
Sobolev space $W^1$. Since $\db$ is a first-order differential
operator, $\Dom(\db)\supset W^1$,
and $\Dom(\db_c)\supset W^1_0$, the closure 
of $C_0^\infty(\Omega)$ in $W^1$.

\begin{lemma} \label{le:W1}
For any bounded domain 
$\Omega\subset\CC^n$, we have $\Dom(\db_{\tilde c})\subset 
W^1(\Omega)$.  If, moreover, $\Omega$ is a Sobolev extension domain, 
then $\Dom(\db_s)\supset W^1(\Omega)$.
\end{lemma}

\begin{proof}
Let $f\in \Dom(\db_{\tilde c})$, and let $f^0$ be its
trivial extension. By definition,
$f\in L^2(\Omega)$ and $\db f^0 \in L^2_{0,1}(\CC^n)$
in the sense of distributions.
Let $\{\phi_\eps\}$ be an approximate identity such that
$\phi_\eps\in C^\infty_0(B_\eps(0))$,
$\phi_\eps\ge 0$  and $\int \phi_\eps =1$.  
We regularize $f$ by  convolution  
$f_\eps=f^0\ast \phi_\eps$. 
Then $f_\eps \in C_0^\infty(\CC^2)$,
and
\begin{equation}\label{eq:f db} f_\eps\to f^0\,,\quad
\db f_\eps \to \db f^0 \qquad \text{in } L^2(\CC^2).
\end{equation}
Using integration by parts, we have  
\begin{equation}\label{eq:f partial} 
\begin{aligned} (\db f_\eps, \db f_\eps) 
&=-\sum_{j=1}^n  \bigg(\frac {\partial^2  
f_\eps}{\partial z_j\partial \overline z_j}, f_\eps \bigg)
=(\partial f_\eps,\partial f_\eps).
\end{aligned}\end{equation}
It follows that $f_\eps\to f^0$ in $W^1(\CC^2)$,
and hence $f\in W^{1}(\Omega)$.

For the second claim, let $f\in W^1(\Omega)$, where
$\Omega$ is an extension domain.
Then there exists a sequence of smooth functions $f_\nu$
on $\overline{\Omega}$ such that $f_\nu$ converges to $f$
in $W^1(\Omega)$. In particular, $f_\nu\to f$ and $\db f_\nu\to\db f$ 
in $L^2(\Omega)$, that is, $f\in\Dom(\db_s)$ and $\db_sf=\db f$.
\end{proof}

\begin{definition}\label{de:Bergman} The {\em Bergman space} 
$\mathcal H(\T)$ is  the closed  subspace  of $L^2(\T)$ consisting of the
holomorphic functions on $\T$, i.e., $\mathcal H(\T)= \Ker(\db) .$ 
The orthogonal projection $$B:L^2(\T)\to \mathcal H(\T)$$ 
is called the {\em Bergman projection}.  
   \end{definition}

We next analyze the kernel of $\db_s$.
Any function $f\in \mathcal H(\T)$
admits a Laurent expansion of the form
\begin{equation}\label{eq:f laurent}
f =\sum_{j=0}^\infty\sum_{k=-1}^\infty 
a_{jk}\left(\frac z w\right)^jw^{k} 
\end{equation}
that converges in $L^2(\T)$.
The functions 
\begin{equation}
\label{eq:ONB}
v_{jk}(z,w):= \left(\frac z w\right)^jw^{k}\,,
\quad j\ge 0, k\ge -1
\end{equation}
that appear in the expansion are pairwise orthogonal, since their
restrictions to any torus $\{|z|=r,|w|=s\}$ agree 
(up to re-labeling and multiplication by 
constants) with a subset of the standard Fourier basis 
$e^{i(\ell\alpha +m\beta)}$.
By Eq.~\eqref{eq:f laurent} they form
a complete orthogonal system for $\Ker(\db)$.
 
\begin{proposition}\label{prop:Ker} On $\T$, 
we have  $\Ker(\db_s)=\Ker(\db)$  on functions.  
\end{proposition}

\begin{proof} We will show that
$\Ker(\db_s)$ contains the functions
$v_{jk}$ from Eq.~\eqref{eq:ONB}.
Since $v_{jk}\in W^1(\T)$ for $j,k\ge 0$,
Lemma~\ref{le:W1} implies that
$$
v_{jk}\in \Ker(\db)\cap W^1(\T)\subset \Ker(\db_s)\,,\qquad j\ge 0, k\ge 0\,.
$$
For $k=-1$, fix $j\ge 0$ and set $u:=v_{j,-1}$.
Given $0<\delta\le 1$, consider the subdomain 
$$\T_\delta :=\left\{(z,w)\in \T\ \big\vert\  |z|<|w|<\delta\right\}\,,
$$
and define the function
\begin{equation}\label{eq:u delta} u_\delta = 
\begin{cases} 
\big(\frac {|w|}{\delta}\big)^\delta u, 
\qquad &\text{on } \T_\delta,\\
u, &\text{on } \T\setminus \T_\delta.
\end{cases}
\end{equation}
Clearly, $|u_\delta|\le |u|$, and $u_\delta\to u$ in
$L^2(\T)$ by dominated convergence.

By construction, $u_\delta$ is piecewise $C^1$.
Its first-order partial derivatives 
are pointwise bounded by $ C|w|^{-2+\delta}$, where $C$ depends on $j$.
For $\delta>0$ this is square integrable, and 
$u_\delta\in W^1(\T)$.  Therefore $u_\delta\in \Dom (\db_s)$
and $\db_s u_\delta =\db u_\delta$.
Since $\db u=0$, we see that
$ \db_s u_\delta =u \, \db\bigl(\tfrac{|w|}{\delta}\bigr)^\delta$
on $\T_\delta$, and vanishes on the complement.
By scaling, 
$$ 
\|\db_s u_\delta\|_{L^2(\T)} 
= \|\db u_\delta\|_{L^2(\T_\delta)}
= \delta\|\db u_1\|_{L^2(\T)}\to 0
$$
as $\delta\to 0$.  Hence $u\in\Dom(\db_s)$ and $\db_s u=0$.

Thus $\Ker(\db_s)$ contains 
an orthonormal basis of $\Ker(\db)$.
Since $\Ker(\db_s)$ is a closed subspace of
$\Ker(\db)$, the two spaces agree.
\end{proof}

\begin{proposition}\label{prop:W=S 2} On 
$\T$, we have $\db_{\tilde c}=\db_c$ on functions.
\end{proposition}

\begin{proof}
Since $\db_{\tilde c}$ is an extension of $\db$,
we have that $\Dom(\db_c)\subset \Dom(\db_{\tilde c})$. 
We now establish the reverse inclusion. 
By Lemma~\ref{le:W1} and the Sobolev Embedding Theorem,
\begin{equation}\label{eq:f W1}\Dom(\db_{\tilde c})
\subset W^1(\T)\subset L^4(\T)
\end{equation}

Given $f\in \Dom(\db_{\tilde c})$,
we approximate $f$ by a function
that vanishes near the singular point at the origin.
Let $\chi_\delta$ be a smooth
cut-off function such that $\chi_\delta=1$ outside the
 ball $B_{2\delta}(0)$, $\chi_\delta$ vanishes on $B_{\delta}(0)$,
and its differential satisfy the pointwise bound
$|d \chi_\delta | \le C\delta^{-1}$ 
where $C$ is a constant independent of $\delta$. 

By the chain rule,
\begin{equation} \db (\chi_\delta f)=(\db \chi_\delta) f+\chi_\delta \db f\,.
\end{equation}
It is clear that 
$\chi_\delta f\to f$ and $\chi_\delta\db f\to \db f$
in $L^2$ as $\delta\to 0$. 

It remains to show that $(\db \chi_\delta) f \to 0$.
By the Cauchy-Schwarz inequality, we have 
$$
\int_{\T} |\db (\chi_\delta) f|^2 dV 
 \le  \bigg(  \int_{B_{2\delta}(0)\cap \T}   |\db\chi_\delta|^4\, 
dV\bigg)^{\frac 12}
\bigg( \int_{B_{2\delta}(0)\cap \T}|f|^4dV\bigg)^{\frac 12}\,.
$$
The first factor is bounded independently of $\delta$.
Since $f\in L^4(\Omega)$, it follows that
$$\|\db(\chi_\delta) f\|^2\le  \tilde C  
\bigg( \int_{B_{2\delta}(0)\cap \T}|f|^4dV\bigg)^{\frac 12} \to 0\,.$$ 
Therefore $\db(\chi_\delta f)\to\db f$ as $\delta\to 0$.
  
We have approximated $f\in\Dom(\db_{\tilde c})$ 
in the graph norm of $\db$
by $ \chi_\delta f$. Since $\chi_\delta f$
is supported
in the bounded Lipschitz domain
$\T  \setminus \overline{B_\delta(0)}$,
it can be further approximated by compactly 
supported functions in $\T$. This proves that $f\in \Dom(\db_c)$. 
\end{proof}

\subsection{Weak equals strong}

We need two more tools, Serre duality  and Dolbeault cohomology.
$L^2$ Serre duality establishes
a relation between $\db$ and $\db_c$, and 
correspondingly between $\db_s$ and $\db_{\tilde c}$.
Denote by
$\star :L^2_{p,q}(\Omega)\to L^2_{n-p,n-q}(\Omega)$ the 
Hodge star operator. 

\begin{lemma}\label{le:dual} Let $\Omega$ be a bounded domain in  
$\CC^n$. Then $\db_c=-\star\db^*\star$.
\end{lemma} 
 \begin{proof}
See  \cite[Proposition 1]{ChakrabartiShaw12} or \cite[Lemma 2.2] {LaurentShaw15}.     
\end{proof}     

\begin{lemma}\label{le:dual T} 
On $\T$, we have $\db_{\tilde c}=-\star \db_s^* \star $.
\end{lemma}
\begin{proof} Since the boundary of $\T$ is rectifiable, the weak 
minimal closure  $\db_{\tilde c}$ is dual to the strong
maximal extension $\db_s$ (see \cite{LaurentShaw13}). 
\end{proof}

\begin{definition}
For $0\le p\le n$ and $0\le q\le n$, the 
{\em $L^2$ Dolbeault cohomology groups}
are  defined by
$$H^{p,q}_{L^2,\db}(\Omega) = \frac {\{f\in L^2_{p,q}(\Omega)\mid \db f=0\} } {\{f\in L^2_{p,q}(\Omega)\mid f=\db u \text{ for some }u\in L^2_{p.q-1}(\Omega)\}}.$$
Similarly, we define $H^{p,q}_{L^2,\db_s}(\Omega)$ by substituting $\db$ 
with $\db_s$. 
\end{definition}

When $\Omega$ is a bounded   
pseudoconvex domain in $\CC^n$, the $L^2$ theory for $\db$ is completely 
known from  H\"ormander's $L^2$ theorem for $\db$ 
(see \cite{Hormander65}). The key result is that
$$ H^{p,q}_{L^2, \db}(\Omega)=0\,,\qquad 1\le p\le n, 1\le q<n\,.
$$

For the strong maximal extension $\db_s$ on
a pseudoconvex domain with rectifiable boundary,
it was proved in~\cite{LaurentShaw13}
that either $H^{0,1}_{L^2, \db_s}(\Omega)=0$, or 
$H^{0,1}_{L^2, \db_s}(\Omega)$ is not Hausdorff. 

 \begin{theorem}\label{th:closed strong} 
On $\T$, the  strong maximal extension $\db_s$ 
of the Cauchy-Riemann operator
has closed range.
\end{theorem}

\begin{proof} 
We will show that $\db_s:L^2_{p,q-1}({\T})\dashrightarrow L^2_{p,q}({\T})$
has closed range for $p=0,1,2$ and $q=1,2$.
As noted above, the value of $p$ plays no role here.

\smallskip 
{\em q=2:} \ Take $p=2$. By Proposition~\ref{prop:W=S 2},
$\db_{\tilde c}=\db_c$ on functions.
By Lemmas~\ref{le:dual} and~\ref{le:dual T}, 
this is equivalent to 
\begin{equation} \label{eq:2,2}
\db_s=\db: L^2_{2,1}(\T)\dashrightarrow L^2_{2,2}(\T)\ .
\end{equation}
 In particular, $\Range(\db_s)=\Range(\db)= L^2_{2,2}(\T)$,
which is closed. 

\smallskip 
{\em q=1:}\  Take $p=0$, and consider 
$\db_s:L^2(\T)\dashrightarrow L^2_{0,1}(\T)$.
By combining Proposition~\ref{prop:Ker} with
Lemma~\ref{le:closed equivalent}, we see that
$$ \Range(\db_s^*)
\subset (\Ker(\db_s))^\perp 
= (\Ker(\db))^\perp 
= \Range(\db^*)\,,$$
where we have used that
${\Range(\db^*)}\subset L^2(\T)$ is closed by H\"ormander's $L^2$-theory.
Since $\db$ is an extension of $\db_s$, we also have the
reverse inclusion
$${\Range(\db_s^*)}\supset \Range(\db^*)\,.  $$ 
Therefore $\Range(\db_s^*)=\Range(\db^*)\subset L^2(\T)$.
By Lemma~\ref{le:closed equivalent}, 
$\db_s :L^2(\T)\dashrightarrow L^2_{0,1}(\T)$ 
has closed range as well.  
\end{proof}
 
\begin{proposition}
\label{prop:vanishing harmonic}  $H^{p,1}_{L^2,\db_s}(\T)=0$ 
for $0\le p\le 2$.
\end{proposition}  
\begin{proof} Take $p=0$.
Since $\db_s:L^2(\T)\dashrightarrow L^2_{0,1}(\T)$
has closed range by Theorem~\ref{th:closed strong},
the corresponding cohomology group
$H^{0,1}_{L^2,\db_s}(\T)$ is Hausdorff 
(see \cite[Proposition 4.5]{Treves67}).  
It follows from~\cite[Theorem 3.2 (iv)]{LaurentShaw13}
that $H^{0,1}_{L^2, \db_s}(\T)=0$. 
 \end{proof}

\begin{theorem} 
\label{th:W=S} 
On $\T$, the  strong maximal extension $\db_s$ 
of the Cauchy-Riemann operator
equals the weak maximal extension $\db$. 
\end{theorem}
\begin{proof}  {\em q=2:}\ By
Eq.~\eqref{eq:2,2}, we have that
$\db_s=\db$ on $(p,1)$-forms for $p=0,1,2$.

\smallskip
{\em q=1:}\ Take $p=0$ and consider 
$\db_s : L^2(\T)\dashrightarrow L^2_{0,1}(\T)$. 
By Proposition~\ref{prop:Ker}
$$
\Ker(\db_s)=\Ker(\db)= \mathcal{H}(\T)
$$ 
on functions. Since $\db=\db_s:L^2_{0,1}\dashrightarrow L^2_{0,2}$
by the first part of the proof, 
we have that $\Ker(\db_s)=\Ker(\db)\subset L^2_{0,1}(\T)$.
By Proposition~\ref{prop:vanishing harmonic}
and H\"ormander's $L^2$ results, 
$H^{0,1}_{L^2,\db_s}= H^{0,1}_{L^2,\db}=0$,
which means by the definition of the cohomology groups that
$$ 
\Range(\db_s)=\Range(\db)\subset L^2_{0,1}(\T).
$$ 
Since $\db$ is a closed extension of the densely defined
operator $\db_s$,
with the same kernel and range, $\db=\db_s$ on functions.
\end{proof}

\begin{corollary}[\textbf{Bergman projection}]\label{co:Bergman} 
Let $B_s: L^2(\T)\to \mathcal H(\T)$ 
be  the Bergman projection with respect to $\db_s$  on $\T$. Then $B=B_s$. 
Moreover, for any $f\in L^2(\T)$, the complementary projection
satisfies
$$ 
f-Bf\ = \ \db_s^* \db_s N_0  f \ = \ \db_s^* N_1  \db_s f\,,
$$ 
where $N_0$ is the $\db$-Neumann operator on functions,
and $N_1$  is the $\db$-Neumann operator on  $(0,1)$-forms.
 
\end{corollary} 
\begin{proof}  
Since $\Ker(\db_s)= \mathcal{H}(\T)=\Ker(\db)$,
either $\db_s$ or $\db$ can be used to define the Bergman projection.
The formulas for the orthogonal projection hold for $\db$ 
by H\"ormander's theory and by Kohn's formula for the Bergman 
projection  (see  Theorem 4.4.3 and Corollary 4.4.4 in 
\cite{ChenShaw01}). Since $\db=\db_s$,   the corollary follows.  
 \end{proof}

\section{\bf Dolbeault cohomology on the complement of $\T$}

In this section, we study the Dolbeault cohomology groups
on an annulus between a pseudoconvex domain and the 
Hartogs triangle $\T$.  

 \begin{definition}
 Let $\Omega$ be a bounded domain in $\CC^n$. 
  Let $W^k(\Omega)$ be the Sobolev  space of order 
$k\in \NN\cup \{0\}$. 
We  denote by $H^{p,q}_{W^k}(\Omega)$ the associated 
cohomology group defined by  
$$ H^{p,q}_{W^k}(\Omega)= 
  \frac {\{f\in W^k_{p,q}(\Omega)\mid \db  f=0\} } {\{f\in W^k_{p,q}(\Omega)\mid f=\db  u \text{ for some }u\in W^k_{p.q-1}(\Omega)\}}.$$
When $k=0$, we also use the notation $ H^{p,q}_{L^2}(\Omega)$ 
to denote the $L^2$ Dolbeault cohomology groups with respect to $\db$. 
\end{definition}
 
We will need the following result
from the book of Chen-Shaw (\cite[Theorem 9.1.3]{ChenShaw01}).  

\begin{lemma} \label{le:compat} 
Let $\Omega$ be a bounded pseudoconvex domain
in $\CC^n$, $n\ge 2$. For any $f\in L^2_{p,q}(\CC^n)$, 
$0\le p\le n$, $1\le q\le n$, such that $f$ is supported
in $\overline\Omega$ and
$$ 
\int_{\Omega}  f\wedge \phi     =0,  
\quad\phi\in L^2_{2-p,0}  (\Omega) \cap\Ker(\db)\,
$$ 
there exists $u\in L^2_{p,q-1}(\CC^n)$
such that $\db_{\tilde c} u=f$.

\end{lemma}

Consider an annular domain
\begin{equation}\label{eq:annulus}\Omega 
=\Omega_1\setminus \overline {\T}
\end{equation}
where  $\Omega_1$ is a pseudoconvex domain in $\CC^2$
containing $\overline\T$.

\begin{theorem}\label{th:W1 L2}
Let $\Omega$ be given by Eq.~\eqref{eq:annulus},  
where $\Omega_1\subset\CC^2$ is a bounded  pseudoconvex domain 
with $C^2$-boundary, and $\overline{\T}\subset \Omega_1$. 
Then 
$$
H^{p,1}_{W^1} (\Omega) \cong \mathcal H(\T)\,,\qquad 0\le p\le 2\,,
$$
where $\mathcal{H}(\T)$ is the Bergman space of $\T$.
In particular, $H^{p,1}_{W^1} (\Omega) $ 
is Hausdorff and infinite-dimensional.
\end{theorem}

\begin{proof} We will prove that
$ H^{p,1}_{W^1} (\Omega) \cong (\mathcal H(\T))'$,
the space of bounded linear forms on the Bergman space.
It suffices to consider $p=2$.

By Lemma~\ref{le:extension annuli},
there is a bounded linear extension operator
$\eta: W^1(\Omega)\to W^1(\Omega_1)$.
For $f\in W^1_{2,1}(\Omega)$, define 
$\ell_f\in (\mathcal H(\T))'$ by 
\begin{equation}
\label{eq:pairing} 
\ell_f(h):= \int_{\T}\db (\eta f)\wedge h, \quad  
h\in \mathcal H(\T)\,.
\end{equation} 
Clearly, $\ell_h(f)$ is bilinear and jointly
continuous in $f\in W^1_{2,1}(\T)$ and 
$h\in \mathcal{H}$.

\smallskip{\em Smooth approximation.}\ 
Let $h\in \mathcal{H(\T)}$. Since $\mathcal{H}(\T)=\Ker(\db_s)$ 
by Proposition \ref{prop:Ker}, there is a sequence
$\{h_\nu\}$ in $C^\infty (\overline{\T})$ such that  
\begin{equation}\label{eq:h graph}\begin{cases} 
&h_\nu\to h,\qquad \text{in }L^2 (\T),
\\&\db h_\nu\to 0\qquad \text{in }L^2_{0,1}(\T).\end{cases}\end{equation}
Since $\db (\eta f)\wedge h = \db(\eta f \wedge h)$,
Stokes' theorem implies that
\begin{equation} \label{eq:Stokes}
\ell_f(h)
=\lim_{\nu\to\infty} \int_{\T} \db(\eta f \wedge h_\nu)
=\lim_{\nu\to\infty} \int_ {b\T} \tau f\wedge h_\nu\,,
\end{equation}
where $\tau f$ is the trace of $f$.
We have used that $\tau f\in L^2(b\T)$ by Corollary~\ref{co:trace}.
It is apparent from Eq.~\eqref{eq:Stokes}
that $\ell_f$ does not depend on the choice of
the extension~$\eta$.

\smallskip {\em $\ell_f$ is determined by
the cohomology class $[f]$.}\ 
Suppose that $f=\db u$ for some $u\in W^1_{2,0}(\Omega)$.
Let $\eta u$ be the extension of $u$ to $W^1(\Omega_1)$.
Since $\Omega_1$ is a Lipschitz domain,
there exists a sequence $u_j\in C^\infty(\CC^2)$ with
$u_j\to \eta u$ in $W^1(\Omega_1)$.
By Eq.~\eqref{eq:Stokes} and two more applications
of Stokes' theorem,
$$
\ell_{\db u_j}(h_\nu)
=\int_{b\T}(\db u_j)\wedge h_\nu
=\int_{b\T}   u_j\wedge\db  h_\nu
= \int_{\T} \db u_j\wedge \db h_\nu\,.
$$
Taking first $j\to\infty$ and then $\nu\to\infty$, we arrive at
$\ell_{\db u}(h)=0$.
Thus the map $[f]\mapsto \ell_f$ is well-defined
from $H^{0,1}_{W^1}(\Omega)$ to $(\mathcal H(\T))'$. 

\smallskip {\em $[f]\mapsto \ell_f$ is injective.}\ 
Suppose that $\ell_f$ vanishes on $\mathcal{H}(\T)$. 
By Lemma~\ref{le:compat}, there exists $g\in L^2_{2,1}(\T)$
such that $\db_{\tilde c}g=\db(\eta f)$ on $\T$.
In fact, the trivial extension $g^0$ of $g$ lies in 
$W^{1}_{2,1}(\CC^2)$. Set $F=\eta f- g^0\in W^1(\Omega_1)$.

By construction, $\db F=0$ on $\Omega$.
Since $\Omega_1$ has $C^2$ boundary, we can solve 
$\db u=F$ for some function $u\in W^1_{2,0}(\Omega_1)$ 
(see \cite{Kohn73} and \cite{Harrington09}).
In particular, $\db u=f$ on $\Omega$.   

\smallskip {\em $[f]\mapsto \ell_f$ is surjective.}\ 
Let $\ell\in (\mathcal H(\T))'$. 
Since    $\mathcal  H(\T)$ is a Hilbert space, $\ell$  
can be represented by some  holomorphic function  
$g\in \mathcal H(\T)$. Let $g^0\in L^2(\Omega_1)$ be the trivial
extension of $g$, and let $\star g^0$ be the
dual $(2,2)$-form on $\Omega_1$. Since  a top degree form  is always 
$\db$-exact, there exists $v\in W^1_{2,1}(\Omega_1)$
that solves $\db v=\star g^0$ on $\Omega_1$.
By construction,
$$
\ell(h)=(g,h)=\int_{b\T} \db v \wedge h\,,\qquad h\in \mathcal{H}(\T)\,.
$$
Let $f$ be the restriction of $v$ to $\Omega$.
Then $f\in W^1_{2,1}(\Omega)$, and $v$ is an extension of $f$ 
to~$\Omega_1$. Since the extension does not matter,
$\ell=\ell_f$.

\smallskip We conclude that $[f]\mapsto \ell_f$ is 
a linear isomorphism from $H^{2,1}_{W^1}(\Omega)$ to 
$\mathcal H(\T)'$.  Since $\mathcal{H}$ is a Hilbert space,
the theorem is proved.
\end{proof}

 \section{\bf Some open questions}   
    
 Let  $\Omega_1$ and $\Omega_2$ be two bounded pseudoconvex domains in $\CC^n$ and let  $\overline {\Omega}_2\subset \Omega_1$. 
  Let $\Omega$ be the annulus    between the two pseudoconvex domains with
  $$\Omega=\Omega_1\setminus \overline {\Omega}_2.$$ 
It is known for $\Omega= \Omega_1\setminus \overline{\T}$  
that the classical Dolbeault cohomology with smooth coefficients on $\Omega$,
   $$H^{0,1}(\Omega):=
  \frac {\{f\in C^\infty_{0,1}(\Omega)\mid \db f=0\} } {\{f\in C^\infty_{0,1}(\Omega)\mid f=\db  u \text{ for some }u\in C^\infty(\Omega)\}}$$  is non-Hausdorff (see \cite[Corollary 4.6]{LaurentShaw13}). 
This is in sharp contrast to Theorem \ref{th:W1 L2}.

  Theorem \ref{th:W1 L2} is a generalization of a result 
by H\"ormander for the case when $\Omega$ is the annulus  between
 two concentric balls in $\CC^n$ (see \cite{Hormander04}.
In that case, $H^{0,n-1}_{L^2}(\Omega)$ is Hausdorff,
and  one can realize  the space   $H_{L^2}^{0,n-1}(\Omega)$
explicitly as the Bergman space of the inner domain.

  When $\Omega_2$ is a pseudoconvex domain with $C^3$ boundary and $0<q<n-1$, the $L^2$ and Sobolev cohomology groups for $\db$ on the annulus were
studied much earlier  in \cite{Shaw85}. 
For general pseudoconvex domains with $C^2$ boundary, the 
Hausdorff property for the critical degree $q=n-1$  is proved in 
 \cite{Shaw10}. The necessary conditions for the Hausdorff properties for the Dolbeault
cohomology group for  $\db$ on  annuli is proved in~\cite{FLS17}.

In view of  Theorem \ref{th:W1 L2} and the remarks 
above, it is natural to ask the following question.   

\begin{problem}\label{prob:1} Let $\Omega=\Omega_1\setminus \overline{\T}$. 
Determine if $H^{0,1}_{L^2}(\Omega)$ is Hausdorff. 
\end{problem}

Without loss of generality, we can take the 
outer domain $\Omega_1$  in  Problem \ref{prob:1}   to be the ball of radius $r\ge 2$ centered at 0 
(see \cite{CLS} for a discussion on this).   
  Problem \ref{prob:1}  can be   called the {\it Dollar Bill} problem since the shape is   featured on the reverse of the American one-dollar bill. 

When the inner domain is the bidisk $D^2$,   
the corresponding problem for  $\Omega=B\setminus D^2$,  
is called  the {\it Chinese Coin} problem since it has the shape of an ancient 
Chinese coin. The Chinese coin  problem is solved in \cite{CLS}.  
Problem~1 has an equivalent formulation
in terms of the $W^1$ Dolbeault cohomology of $\T$:

\begin{proposition}\label{prop:W1 L2}
Let $\Omega= \Omega_1\setminus \overline{\T}$, 
where  $\Omega_1$ is a bounded pseudoconvex domain 
in $\CC^2$  with  $\overline{\T}\subset \Omega_1$. 
Then the following are equivalent:
\begin{enumerate}
\item  $H^{0,1}_{L^2} (\Omega) $ is  Hausdorff;
\item $H^{0,1}_{W^1}(\T)=0$.
\end{enumerate}
\end{proposition} 
 
The proof of Proposition \ref{prop:W1 L2} is the same 
as for Lipschitz domains, as given 
in~\cite[Corollary 4.8]{LaurentShaw13}. The key points
are the $L^2$-duality between $\db_s$ and $\db_{\tilde c}$ 
and the extension property
(Lemma \ref{le:extension annuli}).

\smallskip This leads to a more general question.

\begin{problem} Determine if $ H^{0,1}_{W^s}(\T)=0$, where $s>0$. 
\end{problem}

We remark that if $\Omega$ is a bounded pseudoconvex domain with smooth boundary in $\mathbf C^n$, it follows from \cite{Kohn73} that $ H^{0,1}_{W^s}(\Omega)=0$ for all $s>0$.
Not much is known about
Sobolev estimates for solutions of $\db$ for the
Hartogs triangle.   But there has been  a lot of work for 
$\db$  in other function spaces. 

It is proved in \cite{Sibony} that there is a form 
$f\in C^\infty_{(0,1)}(\overline {\T})$ with 
$\db f=0$ such that  the equation $\db u=f$
has no solution $u\in C^\infty(\overline {\T})$.
Furthermore,  it is proved in \cite{LaurentShaw15} 
that the Dolbeault cohomology with smooth coefficients
on $\T$ is non-Hausdorff. 

On the other hand,
since $\T$ is pseudoconvex, we have from the  Dolbeault  theorem that 
\begin{equation}\label{eq:Dolbeault} 
H^{0,1} (\T) = 0\end{equation}
where  $H^{0,1}(\T)$ denotes the Dolbeault cohomology 
with smooth  $C^\infty(\T)$ coefficients.
Furthermore, there do exist
{\it almost} smooth solutions to the $\db$  
problem on the Hartogs triangle:
  For every $k\in \NN$ and $0<\alpha<1$, let 
$C^{k,\alpha}(\T)$ denote the H\"older continuous 
function spaces of order $k$, $\alpha$. 
    Let $H^{p,q}_{C^{k,\alpha}}(\T)$ denote the Dolbeault cohomology of 
$(p,q)$-forms  with 
$C^{k,\alpha}(\T)$ coefficients. Using the
integral kernel method, it  is proved in \cite{CC} that   
 $$H^{0,1 }_{C^{k,\alpha}}(\T) =0.$$ 
 Notice that the intersection  
$\cap_{k}C^{k,\alpha}(\T)=C^\infty(\overline{\T})$. 
These results show  the subtlety  of such problems on the Hartogs triangle.   
 

\smallskip
 We can also consider the de Rham complex  $d$ on $\T$ instead of $\db$. 
 Let $d$ and  $d_s$ denote the weak and strong maximal extensions from $L^2_{q}(\T)$ to $ L^2_{q+1}(\T).$   
 Consider the $d$-Laplacian 
  $$\Delta=d d^*+ d^*d: L^2_q(\T)\dashrightarrow L^2_q(\T),$$
  where $0\le q\le 4$. 
Similarly, we can consider $\Delta_s=d_sd_s^*+d_s^*d_s$.
We  refer to the paper by H\"ormander  (see \cite{Hormander03}) 
for a historical overview  of the Hodge theorem for domains 
with smooth boundary. The Hodge theorem on Lipschitz domains 
in $\mathbf R^n$ was studied in~\cite{MMS}

 \begin{problem} On the Hartogs triangle $\T$, determine
 \begin{itemize} 
 \item     if the  Hodge theorem holds for $\triangle$ (or $\triangle_s$);
 \item   if the spectrum of $\triangle$ (or $\triangle_s$) 
on forms is discrete;
\item   if   $d=d_s$.  
 \end{itemize}
 \end{problem}
 
Notice that on functions we have   
$d=d_s:L^2(\T)\dashrightarrow L^2_1(\T)$,
since smooth functions are dense in
$\text{Dom}(d) = W^1(\T)$ by Corollary~\ref{co:Sobolev}.
 We can also show that   $d=d_s:L^2_3(\T)\dashrightarrow L^2_4(\T)$ by using   arguments similar to the proof of Proposition \ref{prop:W=S 2}.  
It is not known 
if $d=d_s$  for other degrees.  We  refer the reader to 
\cite{Hormander65} for the identity of weak and strong extensions 
of  general systems of first-order differential operators on
 Lipschitz domains. 

\smallskip
The Neumann problem is the
natural boundary value problem for  
$\Delta : L^2(\T) \dashrightarrow L^2(\T)$
on functions, where $\Delta=d^*d$.
By definition, $u\in Dom(\Delta)$ if and only if
$du\in\Dom(d^*)$, i.e., if there exists
some $f\in L^2(\T)$ such that 
 \begin{equation}\label{eq:Neumann} 
(du, dv)= (f,v) \qquad \text{for all } 
v\in W^1(\T).  
\end{equation}
By taking $v$ to be a smooth test function on $\T$,
we see that $\Delta u=f$ in the sense of distributions.
Moreover, any $f\in\Range(\Delta)$ satisfies $(f, 1)=0$.

Corollary \ref{co:Poincare} directly yields the solution 
of the Neumann problem, by providing
for each $f\in L^2( \T)$ with $(f,1)=0$ a
unique $u\in W^1(\T)$ such that
Eq.~\eqref{eq:Neumann} holds.

To see this, consider the closed subspace
$$ V:=\left\{v\in W^1(\T) \ \big\vert \ (v,1)=0\right\}\,,$$
equipped with the inner product $Q(u,v):=(d u, d v)$.
By the Poincar\'e inequality (Corollary~\ref{co:Poincare}), 
$Q$ is positive definite, hence an inner product on $V$,
and the resulting norm $Q(v,v)^{\frac12}$ is equivalent to
the $W^1$-norm.  The map $v\mapsto (f,v)$ defines a continuous 
linear form on $V$. Since $V$ is a Hilbert space, there exists a unique
$u\in  V$ such that 
$$(f,v)=Q(u,v)=(d u,d v),\quad v\in V\,.
$$
Since $(f,1)=0$ by assumption,  this holds also
for $v=1$, proving
Eq.~\eqref{eq:Neumann}.

For $f\in L^2(\Omega)$, let $f_a$ be the average
of $f$ over $\T$.  The operator  $G_N:L^2(\T)\to L^2(\T)$ that
maps $f$ to the  unique solution of the Neumann
problem $\Delta u=(f\!-\!f_a)$
is called the {\em Neumann operator} on $L^2(\T)$.
Since $\Range(G_N)\subset W^1(\T)$, the
Rellich lemma implies that $G_N$ is compact.
Its spectrum consists of a sequence
of eigenvalues $\mu_j$ of finite multiplicity,
with $\mu_j\to 0$.
Its the eigenvalues are positive (except for the
simple eigenvalue at zero), and $L^2(\T)$ has an orthonormal
basis of eigenvectors.
This implies that $\Delta$ has discrete spectrum 
$\lambda_j=\frac 1\mu_j\to\infty$ on $L^2(\T)$.

We also know that $\Delta=d d^*$ on the top degree ($q=4$) 
has discrete spectrum since it corresponds 
to the Dirichlet problem.  However, it is not known if 
$\Delta=dd^*+d^* d$ on $L^2_q(\T)$  
has closed range when $1\le q\le 3$,
and if the de Rham cohomology is represented
by the harmonic forms.

 \begin{problem} Determine the spectrum of the $\db$-Neumann operator 
$$N_1:L^2_{0,1}(\T) \to L^2_{0,1}(\T)\,.
$$
 \end{problem} 
The operator $N_1$ is not compact  on $L^2_{0,1}(\T)$,
since $\T$ is biholomorphic to a product domain 
(see \cite{CS}). 
  It is not known whether
the spectrum consists of a sequence of eigenvalues
(of possibly infinite multiplicity),
or if continuous spectrum may be present. 
 Since we can express $N_0$ by the formula (see \cite[Theorem 4.4.3]{ChenShaw01})
\begin{equation} N_0= \db^* N_1^2 \db,\end{equation} 
we have  that $N_0:L^2(\T)\to L^2(\T)$ is also non-compact on the orthogonal complement of the Bergman space.

We note that for $q=2$,
the operator $N_2: L^2_{0,2}(\T)\to L^2_{0,2}(\T)$ is compact since 
it corresponds to the Dirichlet problem.  Thus the spectrum
of $N_2$ is discrete.  

\bibliography{survey}
\providecommand{\bysame}{\leavevmode\mathbf Tox
	to3em{\hrulefill}\thinspace}

\end{document}